\documentclass[10pt]{article}
\usepackage{amsmath,amssymb}
\usepackage[all]{xy}

\newtheorem{theorem}{Theorem}

\newtheorem{proposition}[theorem]{Proposition}
\newtheorem{corollary}[theorem]{Corollary}
\newtheorem{remark}{Remark}
\newtheorem{definition}{Definition}

\newenvironment{proof}{\begin{trivlist}\item{}\normalfont\textit{Proof\ }}{\hfill$\square$\end{trivlist}}

\newenvironment{example}{\begin{trivlist}\item{}\normalfont\textbf{Example\ }}{\end{trivlist}}

\newcommand{\E}{\mathcal{E}}
\newcommand{\I}{\mathbf{I}}
\newcommand{\Q}{\mathbf{Q}}

\newcommand{\arr}{{\cdot \rightarrow \cdot}}
\newcommand{\AbGp}{\textbf{AbGp}}
\newcommand{\Cat}{\textbf{Cat}}

\newcommand{\Ch}{\textbf{Ch}}
\newcommand{\cocat}[1][-]{\textbf{CoCat}(#1)}
\newcommand{\colax}{\textrm{colax}}
\newcommand{\coeqrel}[1][-]{\textbf{CoEqRel}(#1)}

\newcommand{\copreord}[1][-]{\textbf{CoPreOrd}(#1)}

\newcommand{\mono}[1][-]{\textbf{Mono}(#1)}
\newcommand{\op}{\textit{op}}
\newcommand{\PreOrd}{\textbf{PreOrd}}

\newcommand{\SAbGp}{\textbf{SAbGp}}

\newcommand{\SSet}{\textbf{SSet}}
\newcommand{\Top}{\textbf{Top}}

\renewcommand{\equiv}{\simeq}
\newcommand{\into}{\rightarrowtail}

\begin{document}


\author{Peter LeFanu Lumsdaine \\ \small \tt plumsdai@andrew.cmu.edu}

\title{A Small Observation on Co-categories}
\date{19 June, 2008}


\maketitle

\begin{abstract}
Various concerns suggest looking for internal co-categories in categories with strong logical structure.  It turns out that in any coherent category $\E$, all co-categories are co-equivalence relations.
\end{abstract}


\begin{definition}Let $\E$ be any category.  An (internal) \emph{co-category} $\Q$ in $\E$ is an internal category in $\E^\op$, i.e.\ objects and morphisms in $\E$
$$ \xymatrix{ Q^0 \ar@/^/[r]^l \ar@/_/[r]_r & Q^1 \ar[l]|i \ar^-q[r] & Q^1 +_{Q^0} Q^1 } $$
such that the following diagrams commute:
$$\qquad \xymatrix{Q^0 \ar^l[r] \ar_l[d] & Q^1 \ar|q[d] & A^0 \ar_r[l] \ar^r[d] \\
Q^1 \ar_-{\nu_1}[r] & Q^1 +_{Q^0} Q^1 & Q^1 \ar^-{\nu_2}[l] } \quad
\xymatrix{Q^1 \ar^-q[r] \ar_q[d] & Q^1 +_{Q^0} Q^1 \ar^{[q,\nu_3]}[d] \\
Q^1 +_{Q^0} Q^1 \ar^-{[\nu_1,q]}[r] & Q^1 +_{Q^0} Q^1 +_{Q^0} Q^1 } $$
$$ \xymatrix{ Q^0 \ar^l[r] \ar_l[dr] & Q^1 \ar^i[d] & Q^0 \ar_r[l] \ar^r[dl] \\
& Q^0} \quad
\xymatrix{ & Q^1 \ar_1[dl] \ar^q[d] \ar^1[dr] & \\
Q^1 & Q^1 +_{Q^0} Q^1 \ar^-{\ [li,1]}[l] \ar_-{[1,ri]\ }[r] & Q^1} \ $$ 

\end{definition}

\begin{definition}A co-category $\Q$ is a \emph{co-preorder} if the maps $l,r$ are jointly epimorphic.

A co-category $\Q$ is a \emph{co-groupoid} if there is a map $s:Q^1 \rightarrow Q^1$ satisfying the duals of the usual identities for the inverse map of a groupoid.

A co-groupoid $\Q$ is a \emph{co-equivalence relation} if it is a co-preorder.
\end{definition}

\begin{remark}In a co-preorder, the co-composition $q$ is uniquely determined by the maps $l,r,i$; likewise, in a co-groupoid, the co-inverse map $s$ is determined by the rest of the structure.
\end{remark}

Together with the obvious maps, these give categories and full inclusions
$$ \xymatrix{ \coeqrel[\E] \ar@{^{(}->}[r] & \copreord[\E] \ar@{^{(}->}[r] & \cocat[\E].} $$

\begin{example}If $\E$ has all (or enough) pushouts and $m:S \into A$ is any monomorphism, then the \emph{co-kernel pair} of $m$ is a co-equivalence relation 
$$ \xymatrix{ A \ar@/^/[rr]^{\nu_1} \ar@/_/[rr]_{\nu_2} & & A +_S A \ar[ll]|{[1_A,1_A]} \ar[rr]|-{[\nu_1,\nu_3]} & & A +_S A +_S A. } $$
This gives the object part of a functor $\mono[\E] \to \coeqrel[\E]$, which (almost by definition) is one half of an equivalence whenever $\E$ is co-exact.  (Here $\mono[\E]$ denotes the full subcategory of $\E^\arr$ on monomorphisms.)
\end{example}

\begin{example}A paradigmatic example is the interval $\I$ in $\Top$, where $I^0$ is a singleton, $I^1$ is the unit interval, $l$ and $r$ are the endpoints, $I^1 +_{I^0} I^1$ is two copies of the interval joined end to end, and $q$ is the obvious ``stretching'' map.  Unfortunately, this is also of course not an actual co-category --- the axioms hold only up to homotopy.  However, it provides a very useful mental picture for the arguments below; and if we delete the interior of the interval, we obtain a genuine co-category. See also the examples below for more versions of the interval.
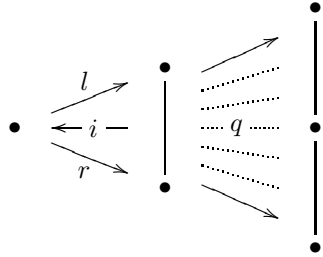
\[ \begin{xy}
(-20,0)*+{\bullet}="p"; 
{ \ar^{\textstyle l} (-15,2)*{} ; (-5,6)*{} }; 
(-9.5,0)*+{i}="i" ; (-5,0)*{} **\dir{-} ;
{ \ar "i" ; (-15,0)*{} }; 
{ \ar_{\textstyle r} (-15,-2)*{} ; (-5,-6)*{} }; 
(0,8)*+{\bullet}="l"; (0,-8)*+{\bullet}="r" **\dir{-};
{\ar (5,7.5)*{} ; (15,12)*{} } ;
{\ar@{.} (5,5)*{} ; (15,8)*{} } ;
{\ar@{.} (5,2.5)*{} ; (15,4)*{} } ;
(5,0)*{} ; (9.5,0)*+{q}="q" **\dir{.}; {\ar@{.} "q" ; (15,0)*{} } ;
{\ar@{.} (5,-2.5)*{} ; (15,-4)*{} } ;
{\ar@{.} (5,-5)*{} ; (15,-8)*{} } ;
{\ar (5,-7.5)*{} ; (15,-12)*{} } ;
(20,-16)*+{\bullet}="l2"; (20,0)*+{\bullet} **\dir{-}; (20,16)*+{\bullet} **\dir{-};
\end{xy} \]
\end{example}

\begin{definition}A \emph{coherent category} is a category with all finite limits, and images and unions that are stable under pullback. 
\end{definition}

\cite[A1.3--4]{eleph} gives various basic results on coherent categories, which we will use here without comment.

\begin{definition}\emph{Coherent logic} is the fragment of first-order logic built up from atomic formul\ae{} using finite con-/dis-junction and existential quantification.
\end{definition}

Coherent logic is discussed in \cite[D1.1--2]{eleph}; the essential point is that coherent logic can be interpreted soundly in coherent categories, and so may be used as an internal language for working in them.

\begin{proposition} \label{prop:cat-is-preord} In a coherent category $\E$, every co-category $\Q$ is a co-equivalence relation.
\end{proposition}

\begin{proof}First, we show that any co-category $\Q$ is a co-preorder.

Arguing in the internal logic: given $x$ in $Q^1$, consider $q(x)$, in $Q^1 +_{Q_0} Q^1$.  Either there is some $y$ in $Q^1$ with $q(x) = \nu_1(y)$, or else some $y$ with $q(x) = \nu_2(y)$.  In the first case, we then have $x = [li,1]q(x) = li(y)$; in the second, $x = ri(y)$.  Thus any $x$ in $Q^1$ is in the image of either $l$ or $r$, i.e.\ $l$ and $r$ are jointly covering, hence epi.  (Indeed, in the first case $x = li(y) = l(il)i(y) = li(li(y)) = li(x)$, and in the second, $x = ri(x)$.)

Restating this diagrammatically: $Q^1 +_{Q^0} Q^1$ is the union of the subobjects $\nu_j: Q^1 \rightarrow Q^1 +_{Q^0} Q^1$, so $Q^1$ is the union of the subobjects $m_j = q^*(\nu_j)$: 
$$\qquad \xymatrix{ P_j \ar[r]^{q_j} \ar[d]^{m_j} & Q^1 \ar[d]^{\nu_j} \\ Q^1 \ar[r]^-q & Q^1 +_{Q^0} Q^1}$$
In particular, the pair $m_1,m_2$ are jointly covering. But by the co-unit identities, $liq_1 = [li,1]\nu_1 q_1 = [li,1]q m_1 = m_1$, and similarly $riq_2 = m_2$.  Thus $liq_1, riq_2$ are jointly covering, and hence so are $l,r$. 

Now, we check that any co-preorder is a co-equivalence relation.  (We give only the diagrammatic version.  Exercise: restate this in the internal logic!)  We want to define $s:Q^1 \rightarrow Q^1$ with $sl = r$, $sr = l$. Since $l,r$ are monos with union $Q^1$, the pullback square
$$\xymatrix{ \bullet \ar[r]^{\pi_2} \ar[d]_{\pi_1} & Q^0 \ar[d]^{r} \\ Q^0 \ar[r]^l & Q^1}$$
is also a pushout, so to construct $s$ as above, it is enough to show that $r \pi_1 = l \pi_2$.  But $ \pi_1 = il\pi_1 = ir\pi_2 = \pi_2$, so $r\pi_1 = r\pi_2 = l \pi_1 = l \pi_2$, and we are done.
\end{proof}

\begin{corollary}If $\E$ is a coherent category with co-kernel pairs of monos, then $\cocat[\E] \equiv \mono[\E]$. (In particular, this holds if $\E$ is a pretopos \cite[A1.4.8]{eleph}.)
\end{corollary}
\begin{proof}
A coherent category is certainly co-effective, so if it has co-kernel pairs, it is co-exact.
\end{proof}

\begin{corollary}For any topos $\E$, $\cocat[\E] \equiv (\E/\Omega)_\colax$.\hfill $\Box$
\end{corollary}

(A \emph{colax} map $(X,\varphi) \to (Y,\psi)$ is a map $f:X \to Y$ such that $\varphi \leq_\Omega \psi f$.)

 In particular, inspecting this equivalence, we see that in this case there is a universal internal co-category in $\E$, from which every co-category in $\E$ may be obtained uniquely by pullback: it is the co-kernel pair of $\top: 1 \rightarrow \Omega$.

\begin{example} The condition that unions are preserved by pullback is crucial: $\AbGp$, for instance, is regular, and has unions, but there is a non-co-preorder co-category corresponding to the interval pictured above, given by the objects
$$Q^0 = \langle v_0 \rangle \qquad Q^1 = \langle v_0, e_1, v_1 \rangle \qquad Q^1 +_{Q^0} Q^1 = \langle v_0, e_1, v_1, e_2, v_2 \rangle$$
(with the natural maps making this a pushout), and maps given by the matrices
$$l = \left( \begin{array}{c} 1 \\ 0 \\ 0 \end{array} \right) \quad
r = \left( \begin{array}{c} 0 \\ 0 \\ 1 \end{array} \right) \quad
i = \left( \begin{array}{c c c} 1 & 0 & 1 \end{array} \right) \quad
q = \left( \begin{array}{c c c} 1 & 0 & 0 \\ 0 & 1 & 0 \\ 0 & 0 & 0 \\ 0 & 1 & 0 \\ 0 & 0 & 1 \end{array} \right) .$$

This example may be given more structure; it is, for instance, the total space of an natural co-category in $\Ch(\AbGp)$.  Since all the underlying groups are free and of finite rank, dualising by transposing matrices also gives corresponding categories in $\AbGp$ and $\Ch(\AbGp)$.
\end{example}

However, any category in a Mal'cev category is a groupoid (this has been observed by various authors, e.g.\ in \cite{carboni-kelly-pedicchio}), so any co-category in a co-Mal'cev category (e.g.\ in an Abelian category, or a topos \cite{bourn-paper}) is a co-groupoid.

\begin{example}An example of a non-co-groupoid co-category is the interval $\I$ in $\Cat$, with $I^0 = ( \cdot ) $, $I^1 = ( \cdot \rightarrow \cdot )$; seen as a co-simplicial object, this is just the usual inclusion functor $\Delta \hookrightarrow \PreOrd \hookrightarrow \Cat$.

Indeed, the functor $\Cat \rightarrow \SSet \rightarrow \SAbGp \rightarrow \Ch(\AbGp) \rightarrow$ $\Ch(\AbGp)$ ``take nerve; take free abelian groups; normalise to a complex; quotient out by subcomplex generated in degrees $\geq$ 2'' sends $\I$ to the co-category in $\Ch(\AbGp)$ of the previous example. 
\end{example}

Co-categories arise as candidate ``interval objects'' when using 2-categories to model intensional type theory \cite{awodey-warren}.  There, one seeks them in categories with some sort of ``weakened'' logical structure; the present result confirms the suspicion that examples in classical ``strict'' logical categories are necessarily fairly trivial.

Many thanks are due to Steve Awodey, for originally posing the question of what co-categories could exist in a topos, to Peter Johnstone, for suggesting and improving parts of the proofs.

\end{document}